\documentclass[11pt]{amsart}

\usepackage[T1]{fontenc}
\usepackage[utf8]{inputenc}
\usepackage{lmodern}
\usepackage{amsmath,amssymb,amsthm,mathtools}
\usepackage{enumitem}
\usepackage{hyperref}

\hypersetup{
  colorlinks=true,
  linkcolor=blue,
  citecolor=blue,
  urlcolor=blue
}

\theoremstyle{plain}
\newtheorem{theorem}{Theorem}[section]
\newtheorem{proposition}[theorem]{Proposition}
\newtheorem{corollary}[theorem]{Corollary}
\newtheorem{lemma}[theorem]{Lemma}

\theoremstyle{definition}
\newtheorem{definition}[theorem]{Definition}
\newtheorem{example}[theorem]{Example}

\theoremstyle{remark}
\newtheorem{remark}[theorem]{Remark}

\newcommand{\Lip}{\operatorname{Lip}}
\newcommand{\lip}{\operatorname{lip}}
\newcommand{\supp}{\operatorname{supp}}

\newcommand{\F}{\mathcal F}
\newcommand{\N}{\mathbb N}
\newcommand{\R}{\mathbb R}

\title[A scalar $c_0$-criterion and gap-coordinate preduals]
{A scalar $c_0$-approximation criterion and gap-coordinate preduals for Lipschitz-free spaces}

\author{Luigi D'Onofrio}
\address{Dipartimento di Scienze e Tecnologie, Universit\`a degli Studi di Napoli Parthenope, Italy}
\email{luigi.donofrio@uniparthenope.it}


\begin{document}

\begin{abstract}
We record a scalar $c_0$-approximation criterion for Banach spaces
canonically embedded into $\ell^\infty$ by a countable norming family.
It identifies the dual of the associated $c_0$-subspace with the prescribed
atomic predual precisely when bounded coordinatewise approximation by
$c_0$-elements is available. The formulation is a self-contained scalar
version of the atomic predual and two-stars framework of
D'Onofrio--Greco--Perfekt--Sbordone--Schiattarella.

We then apply the criterion to Lipschitz-free spaces. The weak-star density
of Lipschitz functions with bounded support is due to
Aliaga--Perneck\'a--Petitjean--Proch\'azka; in proper spaces this yields a
canonical quotient from the bidual of the norm-closed span of compactly
supported Lipschitz functions onto $\Lip_0(M)$.

The main concrete application concerns countable proper subsets
$K\subset [0,\infty)$ with the Euclidean metric. The connected components
$(a_j,b_j)$ of the complement of $K$ give gap coordinates
\[
        \Delta_j f=\frac{f(b_j)-f(a_j)}{b_j-a_j}.
\]
We prove that, in the infinite case, these coordinates identify $\Lip_0(K)$ with $\ell^\infty$,
that the corresponding $c_0$-subspace coincides with Dalet's predual
$S(K)$, and hence that
\[
        S(K)^*\cong \F(K),\qquad S(K)^{**}\cong \Lip_0(K)
\]
canonically and isometrically. This gives a coordinate $c_0$ realization,
for a non-discrete class with accumulation points, of the preduality covered
abstractly by Dalet's theorem for countable proper metric spaces. No new
proof of Dalet's theorem in full generality is claimed.
\end{abstract}

\maketitle

\section{Introduction}

Let $E$ be a Banach space and let $(f_n)_{n\in\N}\subset E^*$ be a
countable norming family. The coordinate map
\[
        V:E\longrightarrow \ell^\infty,\qquad Vx=(f_n(x))_{n\in\N},
\]
is an isometric embedding. It is natural to ask when the subspace
\[
        E_{c_0}:=\{x\in E:Vx\in c_0\}
\]
still determines $E$ by biduality. We give an elementary scalar answer. Let
\[
        V^\sharp:\ell^1\longrightarrow E^*,\qquad
        V^\sharp \xi=\sum_{n=1}^\infty \xi_n f_n,
\]
where $V^\sharp$ is the restriction to the canonical copy of $\ell^1$ in
$(\ell^\infty)^*$ of the Banach-space adjoint of $V$. Set
$Y=V^\sharp(\ell^1)$, endowed with the quotient norm inherited from
$\ell^1/\ker V^\sharp$. Assuming that the canonical map $E\to Y^*$ is an
isometric isomorphism, we prove that
\[
        E_{c_0}^*\cong Y
\]
canonically and isometrically if and only if every element of $E$ is the
coordinatewise limit, along the family $(f_n)$, of a bounded sequence in
$E_{c_0}$. This should be read as a scalar $c_0$-version of the more general
atomic predual and two-stars principle developed in \cite{DGPSS20} for
Banach spaces with supremum-type norms induced by families of operators.

The second part concerns Lipschitz-free spaces. If $(M,d,0)$ is a pointed
metric space, then $\Lip_0(M)$ is canonically the dual of the Lipschitz-free
space $\F(M)$. We use the standard notation
\[
        m_{x,y}=\frac{\delta(x)-\delta(y)}{d(x,y)}\qquad (x\ne y)
\]
for normalized molecules. A lemma of
Aliaga--Perneck\'a--Petitjean--Proch\'azka states that Lipschitz functions
with bounded support are weak-star dense in $\Lip_0(M)$ \cite[Lemma
2.2]{APPP20}. In a proper space, bounded closed sets are compact, so this
implies the weak-star density of compactly supported Lipschitz functions.
We recall the result and record the resulting quotient map
\[
        G_0(M)^{**}\twoheadrightarrow \Lip_0(M),
        \qquad
        G_0(M)=\overline{\Lip_{0,c}(M)}^{\|\cdot\|_{\Lip}}.
\]

The point where the scalar criterion gives a useful additional viewpoint is
coordinate selection. The naive family of all molecules is appropriate in
proper discrete spaces, but not in spaces with accumulation points: long
molecules approaching the same limiting configuration prevent the relevant
predual from being a $c_0$-coordinate space. We therefore introduce a simple
coordinate-admissibility formulation and then treat a non-discrete class
where the coordinates can be chosen explicitly.

Let $K\subset [0,\infty)$ be infinite, countable, closed, and contain the base point
$0$. Such $K$ is a proper metric space; the finite case is trivial. Let $\mathcal G(K)$ be the family of
bounded connected components $(a,b)$ of the complement of $K$ in its convex
hull. Since $K$ is countable, it has Lebesgue measure zero, and the gaps in
$\mathcal G(K)$ carry all one-dimensional length. For each gap $I=(a,b)$ we
consider the coordinate
\[
        \Delta_I f=\frac{f(b)-f(a)}{b-a}.
\]
We prove that
\[
        \|f\|_{\Lip}=\\sup_{I\in\mathcal G(K)} |\Delta_I f|,
\]
that $\F(K)$ is the $\ell^1$-sum of the corresponding gap molecules, and
that the subspace defined by $(\Delta_I f)_{I\in\mathcal G(K)}\in c_0$
coincides with Dalet's predual
\[
        S(K)=\left\{ f\in\lip_0(K):
        \lim_{r\to\infty}
        \sup_{\substack{x\text{ or }y\notin B(0,r)\\ x\ne y}}
        \frac{|f(x)-f(y)|}{|x-y|}=0
        \right\}.
\]
Consequently
\[
        S(K)^*\cong \F(K),\qquad S(K)^{**}\cong \Lip_0(K)
\]
isometrically. Dalet proved in full generality that $\F(M)$ is a dual space
and has the metric approximation property whenever $M$ is countable proper
\cite{Dal15b}; the present result does not extend that theorem. Rather, it
shows that for countable proper subsets of the real line, including spaces
with accumulation points, Dalet's predual is exactly a natural gap-coordinate
$c_0$-space.

All Banach spaces are real. Throughout the paper $B(x,r)$ denotes the
closed ball.

\section{\texorpdfstring{A scalar $c_0$-approximation criterion}{A scalar c0-approximation criterion}}

\begin{definition}\label{def:scalar-setting}
Let $E$ be a Banach space and let $(f_n)_{n\in\N}\subset E^*$ satisfy
\[
        \sup_{n\in\N}\|f_n\|\le 1,
        \qquad
        \|x\|=\sup_{n\in\N}|f_n(x)|\quad (x\in E).
\]
Define
\[
        V:E\longrightarrow \ell^\infty,
        \qquad
        Vx=(f_n(x))_{n\in\N}.
\]
Then $V$ is an isometric embedding. Define also
\[
        V^\sharp:\ell^1\longrightarrow E^*,
        \qquad
        V^\sharp\xi=\sum_{n=1}^\infty \xi_n f_n.
\]
Equivalently, $V^\sharp$ is the restriction of the adjoint
$V^*:(\ell^\infty)^*\to E^*$ to the canonical subspace
$\ell^1\subset (\ell^\infty)^*$.

Let
\[
        Y:=V^\sharp(\ell^1)\subset E^*,
\]
e equip $Y$ with the quotient norm inherited from $\ell^1/\ker V^\sharp$.
Finally, set
\[
        E_{c_0}:=\{x\in E:Vx\in c_0\}.
\]
\end{definition}

The space $E_{c_0}$ is closed in $E$. Indeed,
$V(E_{c_0})=V(E)\cap c_0$, and $V(E)$ is closed in $\ell^\infty$ because
$V$ is an isometry.

The canonical evaluation map
\[
        J:E\longrightarrow Y^*,\qquad Jx(y)=y(x),
\]
is contractive. In the criterion below it is assumed to be onto and
isometric.

\begin{theorem}[Scalar $c_0$-criterion]\label{thm:scalar-c0}
Assume, in the setting of Definition~\ref{def:scalar-setting}, that
\[
        J:E\longrightarrow Y^*
\]
is an isometric isomorphism. Then the following assertions are equivalent.

\begin{enumerate}[label=\textup{(\roman*)}]
\item The canonical restriction map
\[
        \Lambda:Y\longrightarrow E_{c_0}^*,
        \qquad
        \Lambda(y)=y|_{E_{c_0}},
\]
is an isometric isomorphism.

\item For every $x\in E$ there exists a sequence $(x_k)\subset E_{c_0}$
such that
\[
        \sup_k\|x_k\|<\infty
        \qquad\text{and}\qquad
        f_n(x_k)\longrightarrow f_n(x)
        \quad(n\in\N).
\]
\end{enumerate}

When these conditions hold, the sequence in \textup{(ii)} may be chosen
with $\sup_k\|x_k\|\le \|x\|$. Moreover, on bounded sequences in $E$,
coordinatewise convergence on the family $(f_n)$ is equivalent to
convergence in $\sigma(E,Y)$.
\end{theorem}

\begin{proof}
Set
\[
        W:=V(E_{c_0})=V(E)\cap c_0\subset c_0.
\]
The restriction $V:E_{c_0}\to W$ is an isometric isomorphism. Hence
\[
        E_{c_0}^*\cong W^*.
\]
Since $W$ is a closed subspace of $c_0$, the classical duality of subspaces
of $c_0$ gives
\[
        W^*\cong \ell^1/W^\perp,
\]
where
\[
        W^\perp=\left\{\xi\in\ell^1:
        \sum_{n=1}^\infty \xi_n w_n=0\text{ for every }w=(w_n)\in W
        \right\}.
\]
On the other hand, by definition of the quotient norm on $Y$,
\[
        Y\cong \ell^1/\ker V^\sharp.
\]
We first observe that
\[
        \ker V^\sharp\subset W^\perp.
\]
Indeed, if $\xi\in\ker V^\sharp$ and $w=Vx\in W$ with $x\in E_{c_0}$, then
\[
        \sum_{n=1}^\infty \xi_n w_n
        =\sum_{n=1}^\infty \xi_n f_n(x)
        =V^\sharp\xi(x)=0.
\]
Therefore the canonical quotient map
$\ell^1/\ker V^\sharp\to \ell^1/W^\perp$ is contractive; it is precisely
$\Lambda:Y\to E_{c_0}^*$ under the identifications above.

Assume now \textup{(ii)}. Let $\xi\in W^\perp$ and let $x\in E$. Choose
$(x_k)\subset E_{c_0}$ as in \textup{(ii)}, say $\sup_k\|x_k\|\le M$. For
every $k$,
\[
        0=\sum_{n=1}^\infty \xi_n f_n(x_k),
\]
because $Vx_k\in W$ and $\xi$ annihilates $W$. Since
\[
        |\xi_n f_n(x_k)|\le M|\xi_n|,
        \qquad \xi=(\xi_n)\in \ell^1,
\]
the dominated convergence theorem for series yields
\[
        \sum_{n=1}^\infty \xi_n f_n(x)
        =\lim_{k\to\infty}\sum_{n=1}^\infty \xi_n f_n(x_k)=0.
\]
Thus $V^\sharp\xi=0$, and therefore $W^\perp\subset \ker V^\sharp$.
Consequently
\[
        W^\perp=\ker V^\sharp.
\]
It follows that
\[
        E_{c_0}^*\cong W^*\cong \ell^1/W^\perp
        =\ell^1/\ker V^\sharp\cong Y
\]
canonically and isometrically. This proves \textup{(i)}.

Conversely, assume \textup{(i)}. Then
$\Lambda:Y\to E_{c_0}^*$ is an isometric isomorphism, and hence its adjoint
\[
        \Lambda^*:E_{c_0}^{**}\longrightarrow Y^*
\]
is an isometric isomorphism. Let $x\in E$ and set
\[
        \widetilde x=(\Lambda^*)^{-1}(Jx)\in E_{c_0}^{**}.
\]
Then $\|\widetilde x\|=\|x\|$. By Goldstine's theorem, there exists a net
$(u_\alpha)\subset E_{c_0}$ such that
\[
        \sup_\alpha\|u_\alpha\|\le \|x\|,
        \qquad
        \kappa_{E_{c_0}}(u_\alpha)\xrightarrow{w^*}\widetilde x
        \quad\text{in }E_{c_0}^{**}.
\]
For $u\in E_{c_0}$ and $y\in Y$ one has
\[
        \langle \Lambda^*\kappa_{E_{c_0}}(u),y\rangle
        =\langle \kappa_{E_{c_0}}(u),\Lambda y\rangle
        =(\Lambda y)(u)=y(u)=\langle Ju,y\rangle.
\]
Hence $\Lambda^*\kappa_{E_{c_0}}(u)=Ju$. Applying $\Lambda^*$ to the
convergent net gives
\[
        Ju_\alpha\xrightarrow{\sigma(Y^*,Y)}Jx.
\]

Since $Y$ is a quotient of $\ell^1$, it is separable. Let $(y_j)_{j\ge1}$
be dense in the unit ball of $Y$. We choose a sequence from the net as
follows. For each $k$ choose an index $\alpha_k$ such that
\[
        |(Ju_{\alpha_k}-Jx)(y_j)|<\frac1k,
        \qquad 1\le j\le k.
\]
This is possible because the net converges to $Jx$ in $\sigma(Y^*,Y)$.
Set $x_k=u_{\alpha_k}$. Then $\sup_k\|x_k\|\le\|x\|$. If $y\in Y$ and
$\varepsilon>0$, choose $j$ such that $\|y-y_j\|$ is sufficiently small;
using the uniform boundedness of $(Jx_k)$ and $Jx$, one obtains
\[
        (Jx_k)(y)\longrightarrow (Jx)(y).
\]
Thus $x_k\to x$ in $\sigma(E,Y)$. Since each $f_n$ belongs to $Y$, this
implies
\[
        f_n(x_k)\longrightarrow f_n(x)\qquad(n\in\N),
\]
and proves \textup{(ii)}.

It remains to justify the final assertion. The linear span of $(f_n)$ is
dense in $Y$, because finitely supported sequences are dense in $\ell^1$ and
$Y$ is the quotient image of $\ell^1$. If $(z_k)$ is bounded in $E$ and
$f_n(z_k)\to f_n(z)$ for every $n$, then convergence holds first for finite
linear combinations of the $f_n$ and then for every $y\in Y$ by density and
boundedness. Thus $z_k\to z$ in $\sigma(E,Y)$. The converse is immediate.
\end{proof}

\begin{corollary}\label{cor:biduality}
Under the equivalent conditions of Theorem~\ref{thm:scalar-c0}, one has a
canonical isometric identification
\[
        E_{c_0}^{**}\cong E.
\]
\end{corollary}

\begin{proof}
By Theorem~\ref{thm:scalar-c0}, $E_{c_0}^*\cong Y$ isometrically. Taking
duals gives $E_{c_0}^{**}\cong Y^*$. Since $J:E\to Y^*$ is an isometric
isomorphism, the result follows.
\end{proof}

\section{Compact support and a canonical quotient}

Let $(M,d,0)$ be a pointed metric space. We denote by $\Lip_0(M)$ the
Banach space of all real-valued Lipschitz functions $f:M\to\R$ such that
$f(0)=0$, equipped with the norm
\[
        \|f\|_{\Lip}=\sup_{x\ne y}\frac{|f(x)-f(y)|}{d(x,y)}.
\]
The Lipschitz-free space over $M$ is
\[
        \F(M)=\overline{\operatorname{span}}\{\delta(x):x\in M\}
        \subset \Lip_0(M)^*,
\]
where $\langle f,\delta(x)\rangle=f(x)$. It is the canonical predual of
$\Lip_0(M)$:
\[
        \Lip_0(M)=\F(M)^*
\]
isometrically. For $x\ne y$ we write
\[
        m_{x,y}=\frac{\delta(x)-\delta(y)}{d(x,y)}\in\F(M)
\]
for the elementary normalized molecule.

We use the notation
\[
        \Lip_{0,bs}(M)=\{f\in\Lip_0(M):\supp(f)\text{ is bounded}\}
\]
and
\[
        \Lip_{0,c}(M)=\{f\in\Lip_0(M):\supp(f)\text{ is compact}\}.
\]
Here $\supp(f)$ denotes the closed support.

\begin{theorem}[Aliaga--Perneck\'a--Petitjean--Proch\'azka]
\label{thm:APPP}
For every pointed metric space $M$, the space $\Lip_{0,bs}(M)$ is weak-star
dense in $\Lip_0(M)$, that is,
\[
        \overline{\Lip_{0,bs}(M)}^{\sigma(\Lip_0(M),\F(M))}=\Lip_0(M).
\]
\end{theorem}

\begin{proof}
This is \cite[Lemma 2.2]{APPP20}.
\end{proof}

For proper metric spaces, bounded closed sets are compact. Hence
Theorem~\ref{thm:APPP} implies weak-star density of compactly supported
Lipschitz functions. We also recall a norm-controlled cut-off proof, since
it is useful below.

\begin{lemma}[Norm-preserving compactly supported cut-off]
\label{lem:cutoff}
Let $(M,d,0)$ be a pointed proper metric space and let $f\in\Lip_0(M)$.
Put $L=\|f\|_{\Lip}$. For every $R>0$ there exists
$f_R\in\Lip_{0,c}(M)$ such that
\begin{enumerate}[label=\textup{(\arabic*)}]
\item $f_R=f$ on $B(0,R)$;
\item $f_R=0$ on $M\setminus B(0,2R)$;
\item $\|f_R\|_{\Lip}\le \|f\|_{\Lip}$.
\end{enumerate}
\end{lemma}

\begin{proof}
Let
\[
        A_R=B(0,R)\cup (M\setminus B(0,2R)).
\]
Define $g_R:A_R\to\R$ by
\[
        g_R(x)=
        \begin{cases}
        f(x),& x\in B(0,R),\\
        0,& x\in M\setminus B(0,2R).
        \end{cases}
\]
The two pieces are disjoint. The map $g_R$ is $L$-Lipschitz on $A_R$. The
only non-trivial case is when $x\in B(0,R)$ and $y\in M\setminus B(0,2R)$.
Then
\[
        |g_R(x)-g_R(y)|=|f(x)|\le Ld(x,0)\le LR,
\]
while
\[
        d(x,y)\ge d(y,0)-d(x,0)\ge 2R-R=R.
\]
Thus $|g_R(x)-g_R(y)|\le Ld(x,y)$.

By the McShane extension theorem \cite{McS34}, $g_R$ admits an
$L$-Lipschitz extension $f_R:M\to\R$. Since $0\in B(0,R)$ and $f(0)=0$, one
has $f_R(0)=0$. Moreover $f_R=f$ on $B(0,R)$ and $f_R=0$ on
$M\setminus B(0,2R)$. Hence $\supp(f_R)\subset B(0,2R)$, which is compact
because $M$ is proper. Therefore $f_R\in\Lip_{0,c}(M)$ and
$\|f_R\|_{\Lip}\le L$.
\end{proof}

\begin{corollary}\label{cor:compact-density}
If $(M,d,0)$ is proper, then $\Lip_{0,c}(M)$ is weak-star dense in
$\Lip_0(M)$. Moreover, for every $f\in\Lip_0(M)$ the approximating net may
be chosen in $\|f\|_{\Lip}B_{\Lip_0(M)}$.
\end{corollary}

\begin{proof}
Let $f_R$ be the functions given by Lemma~\ref{lem:cutoff}. If
$\mu=\sum_{j=1}^N a_j\delta(x_j)$ is finitely supported and $R$ is large
enough that all $x_j$ belong to $B(0,R)$, then
$\langle f_R,\mu\rangle=\langle f,\mu\rangle$. Since finitely supported
elements are dense in $\F(M)$ and $(f_R)$ is norm bounded, it follows that
$f_R\to f$ in $\sigma(\Lip_0(M),\F(M))$.
\end{proof}

Let
\[
        G_0(M):=\overline{\Lip_{0,c}(M)}^{\|\cdot\|_{\Lip}}
        \subset \Lip_0(M).
\]

\begin{proposition}\label{prop:G0-quotient}
Let $(M,d,0)$ be proper. The restriction map
\[
        R:\F(M)\longrightarrow G_0(M)^*,
        \qquad R\mu=\mu|_{G_0(M)},
\]
is an isometric embedding. Consequently, its adjoint
\[
        R^*:G_0(M)^{**}\longrightarrow \F(M)^*=\Lip_0(M)
\]
is a contractive weak-star continuous quotient map. Moreover, $R^*$ extends
the canonical inclusion $G_0(M)\hookrightarrow \Lip_0(M)$.
\end{proposition}

\begin{proof}
For $\mu\in\F(M)$,
\[
        \|R\mu\|=\sup_{g\in B_{G_0(M)}}|\langle g,\mu\rangle|.
\]
By Corollary~\ref{cor:compact-density}, the unit ball of $G_0(M)$ is
weak-star dense in the unit ball of $\Lip_0(M)$. Since
$f\mapsto \langle f,\mu\rangle$ is weak-star continuous on $\Lip_0(M)$,
\[
        \|R\mu\|=\sup_{f\in B_{\Lip_0(M)}}|\langle f,\mu\rangle|
        =\|\mu\|_{\F(M)}.
\]
Thus $R$ is an isometric embedding. The adjoint of an isometric embedding is
a quotient map of norm one. Finally, if $g\in G_0(M)$ and
$\kappa_{G_0(M)}:G_0(M)\to G_0(M)^{**}$ denotes the canonical embedding,
then, for every $\mu\in\F(M)$,
\[
        \langle R^*\kappa_{G_0(M)}(g),\mu\rangle
        =\langle \kappa_{G_0(M)}(g),R\mu\rangle
        =\langle g,\mu\rangle.
\]
Hence $R^*\kappa_{G_0(M)}(g)=g$ in $\Lip_0(M)$.
\end{proof}

\begin{remark}
The preceding quotient statement is a consequence of the weak-star density
of compact support. It does not imply, and is not meant to imply, a general
identification $\lip_0(M)^{**}\cong\Lip_0(M)$. Compactly supported Lipschitz
functions need not be little Lipschitz functions; for instance, on
$[0,1]$ the function $x\mapsto x(1-x)$ is compactly supported but its
incremental quotients along $(1/k,0)$ tend to $1$. The relationship between compactly supported Lipschitz functions
and the little Lipschitz space on compact metric spaces is
further analysed in \cite{AADM20}, where atomic decompositions
of the predual and biduality results for the pair
$(\mathrm{lip}(K,\rho),\mathrm{Lip}(K,\rho))$ are obtained
in the framework of o-O structures.
\end{remark}

\section{Coordinate families in Lipschitz-free spaces}

We isolate the coordinate condition needed to apply
Theorem~\ref{thm:scalar-c0} to Lipschitz-free spaces.

\begin{definition}\label{def:coordinate-quotient}
Let $(M,d,0)$ be a pointed metric space. A sequence
$\mu=(\mu_n)_{n\in\N}\subset B_{\F(M)}$ is called a \emph{metric quotient
coordinate family} if the summation operator
\[
        T_\mu:\ell^1\longrightarrow \F(M),
        \qquad
        T_\mu((a_n))=\sum_{n=1}^\infty a_n\mu_n,
\]
is a metric quotient map onto $\F(M)$.

For such a family define
\[
        E_{c_0}(\mu)=\left\{f\in\Lip_0(M):
        (\langle f,\mu_n\rangle)_{n\in\N}\in c_0\right\}.
\]
We say that $\mu$ is \emph{$c_0$-admissible} if, for every
$f\in\Lip_0(M)$, there exists a sequence $(f_k)\subset E_{c_0}(\mu)$ such
that
\[
        \sup_k\|f_k\|_{\Lip}\le \|f\|_{\Lip}
\]
and
\[
        \langle f_k,\mu_n\rangle\longrightarrow \langle f,
        \mu_n\rangle
        \qquad(n\in\N).
\]
\end{definition}

If $\mu$ is a metric quotient coordinate family, then its coordinates are
isometrically norming for $\Lip_0(M)$. Indeed,
\[
        \|f\|_{\Lip}
        =\|f\|_{\F(M)^*}
        =\|T_\mu^*f\|_{\ell^\infty}
        =\sup_n |\langle f,
        \mu_n\rangle|.
\]

\begin{proposition}\label{prop:admissible}
Let $(M,d,0)$ be a pointed metric space and let
$\mu=(\mu_n)\subset B_{\F(M)}$ be a $c_0$-admissible metric quotient
coordinate family. Then the canonical restriction map
\[
        \F(M)\longrightarrow E_{c_0}(\mu)^*,
        \qquad
        \eta\longmapsto \eta|_{E_{c_0}(\mu)},
\]
is an isometric isomorphism. Consequently,
\[
        E_{c_0}(\mu)^{**}\cong \Lip_0(M)
\]
canonically and isometrically.
\end{proposition}

\begin{proof}
Set $E=\Lip_0(M)$ and define
\[
        f_n(g)=\langle g,\mu_n\rangle,
        \qquad g\in E.
\]
The preceding observation shows that $(f_n)$ is a countable norming family.
Moreover, the space $Y$ associated with $(f_n)$ in
Definition~\ref{def:scalar-setting} is canonically $\F(M)$, because
$T_\mu$ is a metric quotient map onto $\F(M)$. Under this identification
$J:E\to Y^*$ is precisely the canonical isometric identification
$\Lip_0(M)=\F(M)^*$.

The $c_0$-subspace associated with the coordinates $(f_n)$ is exactly
$E_{c_0}(\mu)$, and $c_0$-admissibility is precisely condition \textup{(ii)}
of Theorem~\ref{thm:scalar-c0}. The conclusion follows from
Theorem~\ref{thm:scalar-c0} and Corollary~\ref{cor:biduality}.
\end{proof}

\begin{corollary}[A compact-support sufficient condition]
\label{cor:compact-admissible}
Let $(M,d,0)$ be proper and let $\mu=(\mu_n)\subset B_{\F(M)}$ be a metric
quotient coordinate family. If
\[
        \Lip_{0,c}(M)\subset E_{c_0}(\mu),
\]
then $\mu$ is $c_0$-admissible. Hence
\[
        E_{c_0}(\mu)^*\cong \F(M),
        \qquad
        E_{c_0}(\mu)^{**}\cong \Lip_0(M)
\]
canonically and isometrically.
\end{corollary}

\begin{proof}
Let $f\in\Lip_0(M)$. By Lemma~\ref{lem:cutoff} and
Corollary~\ref{cor:compact-density}, there is a net, and in particular a
coordinatewise approximating family, $(f_R)\subset\Lip_{0,c}(M)$ such that
$\|f_R\|_{\Lip}\le\|f\|_{\Lip}$ and
$f_R\to f$ in $\sigma(\Lip_0(M),\F(M))$. Since each $\mu_n$ belongs to
$\F(M)$,
\[
        \langle f_R,\mu_n\rangle\to\langle f,
        \mu_n\rangle
        \qquad(n\in\N).
\]
The assumption gives $f_R\in E_{c_0}(\mu)$. Choosing a sequence from the
net by the same separability argument as in Theorem~\ref{thm:scalar-c0}
proves $c_0$-admissibility, and Proposition~\ref{prop:admissible} applies.
\end{proof}

\begin{remark}
Corollary~\ref{cor:compact-admissible} explains why the family of all
molecules works in proper discrete spaces: compact sets are finite there.
It also explains why the same family is too large in the presence of
accumulation points. A Lipschitz function may have non-zero macroscopic
slope along infinitely many molecules converging to the same limiting
molecule, so the corresponding coordinate sequence need not lie in $c_0$.
The next section uses a smaller, adapted family of coordinates.
\end{remark}

\section{Gap coordinates for countable proper subsets of the real line}

We now give the main concrete application. It goes beyond the discrete case
while remaining explicit enough to be proved directly from the scalar
criterion.

Let $K\subset [0,\infty)$ be an infinite countable closed subset containing $0$. With the Euclidean metric and base point $0$, the space
$K$ is proper. Let
\[
        R_K:=\sup K\in (0,\infty]
\]
and put
\[
        I_K=
        \begin{cases}
        [0,R_K],& R_K<\infty,\\
        [0,\infty),& R_K=\infty.
        \end{cases}
\]
Let $\mathcal G(K)$ be the family of bounded connected components of
$I_K\setminus K$. Each $I\in\mathcal G(K)$ is an open interval
$I=(a_I,b_I)$ with $a_I,b_I\in K$ and length
\[
        |I|=b_I-a_I>0.
\]
Since $K$ is countable, it has Lebesgue measure zero; hence the intervals in
$\mathcal G(K)$ cover $I_K$ up to a null set. We fix once and for all an
enumeration
\[
        \mathcal G(K)=\{I_j:j\in\N\}.
\]
For $f\in\Lip_0(K)$ define the gap slopes
\[
        \Delta_I f=\frac{f(b_I)-f(a_I)}{b_I-a_I}
        \qquad(I\in\mathcal G(K)).
\]

\begin{lemma}[Gap decomposition of the free space]
\label{lem:gap-FK}
For $I=(a_I,b_I)\in\mathcal G(K)$ put
\[
        \mu_I:=m_{b_I,a_I}
        =\frac{\delta(b_I)-\delta(a_I)}{b_I-a_I}
        \in\F(K).
\]
Then the map
\[
        U:\ell^1(\mathcal G(K))\longrightarrow \F(K),
        \qquad
        U((c_I))=\sum_{I\in\mathcal G(K)}c_I\mu_I,
\]
is an isometric isomorphism onto $\F(K)$. Consequently,
\[
        \|f\|_{\Lip}=
        \sup_{I\in\mathcal G(K)} |\Delta_I f|
        \qquad(f\in\Lip_0(K)).
\]
\end{lemma}

\begin{proof}
We use the standard identification of $\F([0,R])$ with $L^1[0,R]$, and of
$\F([0,\infty))$ with $L^1[0,\infty)$, under which
\[
        \delta(t)\longmapsto \mathbf 1_{[0,t]}
        \qquad(t\ge0);
\]
see, for example, \cite{Wea18}. The canonical linearization embeds
$\F(K)$ isometrically as the closed linear span of
$\{\mathbf 1_{[0,t]}:t\in K\}$ in the corresponding $L^1$ space.

More precisely, the isometric embedding $F([0,\infty))\hookrightarrow L^1[0,\infty)$
restricts to an isometric embedding of $F(K)$ as follows.
The inclusion $\iota\colon K\hookrightarrow [0,\infty)$ induces a
norm-one restriction map $\iota^*\colon \Lip_0([0,\infty))\to\Lip_0(K)$
and a canonical isometric embedding
$\iota_*\colon F(K)\hookrightarrow F([0,\infty))$
satisfying $\iota_*(\delta_K(t))=\delta(t)$ for $t\in K$
(see \cite[Proposition~2.2.2]{Wea18}).
Under the identification $F([0,\infty))\cong L^1[0,\infty)$,
the image of $\iota_*$ is precisely the closed linear span of
$\{1_{[0,t]}:t\in K\}$ in $L^1[0,\infty)$, and the embedding is isometric
because $\iota_*$ is an isometric embedding and the identification
$F([0,\infty))\cong L^1[0,\infty)$ is an isometric isomorphism.

Under this identification,
\[
        \mu_I\longmapsto \frac{\mathbf 1_I}{|I|}.
\]
The intervals $I\in\mathcal G(K)$ are pairwise disjoint. Therefore, for
$(c_I)\in\ell^1(\mathcal G(K))$,
\[
        \left\|\sum_I c_I\frac{\mathbf 1_I}{|I|}\right\|_{L^1}
        =\sum_I |c_I|.
\]
Thus $U$ is an isometry into $\F(K)$.

It remains to see that the range is all of $\F(K)$. Let $t\in K$. Since
$K$ has measure zero,
\[
        \mathbf 1_{[0,t]}
        =\sum_{\substack{I\in\mathcal G(K)\\ I\subset (0,t)}}
        \mathbf 1_I
        =\sum_{\substack{I\in\mathcal G(K)\\ I\subset (0,t)}}
        |I|\frac{\mathbf 1_I}{|I|}
\]
in $L^1$, and the coefficients have sum $t<\infty$. Hence
$\delta(t)$ belongs to the closed span of the $\mu_I$. Since the elements
$\delta(t)$ span a dense subspace of $\F(K)$, $U$ is onto.

The norm identity for $\Lip_0(K)=\F(K)^*$ follows by duality, because the
unit ball of $\F(K)$ is the unit ball of the $\ell^1$-sum generated by the
basis $(\mu_I)$.
\end{proof}

Define the gap $c_0$-space
\[
        E_{\rm gap}(K):=
        \left\{f\in\Lip_0(K):
        (\Delta_{I_j}f)_{j\in\N}\in c_0\right\}.
\]
By Lemma~\ref{lem:gap-FK}, the map
\[
        \Gamma:\Lip_0(K)\longrightarrow \ell^\infty(\mathcal G(K)),
        \qquad
        \Gamma f=(\Delta_I f)_{I\in\mathcal G(K)},
\]
is an isometric isomorphism. In particular $E_{\rm gap}(K)$ is isometric to
$c_0(\mathcal G(K))$.

\begin{proposition}[Gap coordinates are $c_0$-admissible]
\label{prop:gap-admissible}
The family $(\mu_{I_j})_{j\in\N}$ is a $c_0$-admissible metric quotient
coordinate family for $K$. Consequently,
\[
        E_{\rm gap}(K)^*\cong \F(K),
        \qquad
        E_{\rm gap}(K)^{**}\cong \Lip_0(K)
\]
canonically and isometrically.
\end{proposition}

\begin{proof}
By Lemma~\ref{lem:gap-FK}, the summation map associated with
$(\mu_{I_j})$ is an isometric isomorphism from $\ell^1$ onto $\F(K)$, hence
in particular a metric quotient map.

Let $f\in\Lip_0(K)$ and write
\[
        \alpha_j=\Delta_{I_j}f.
\]
For $N\in\N$ define a bounded scalar sequence
\[
        \alpha^{(N)}_j=
        \begin{cases}
        \alpha_j,& j\le N,\\
        0,& j>N.
        \end{cases}
\]

The sequence $(\alpha^{(N)}_j)_{j\in\mathbb{N}}$ belongs to $\ell^\infty(G(K))$,
since $\sup_j|\alpha^{(N)}_j|\le\sup_j|\alpha_j|=\|f\|_{\mathrm{Lip}}<\infty$.
Since $\Gamma\colon \Lip_0(K)\to\ell^\infty(G(K))$ is an isometric isomorphism
(Lemma~5.1), there exists a unique $f_N\in\Lip_0(K)$ such that
$\Gamma f_N = (\alpha^{(N)}_j)_{j\in\mathbb{N}}$, i.e.\
$\Delta_{I_j}f_N = \alpha^{(N)}_j$ for all $j\in\mathbb{N}$,
and $\|f_N\|_{\mathrm{Lip}} = \|(\alpha^{(N)}_j)\|_{\ell^\infty} \le \|f\|_{\mathrm{Lip}}$.
Then $f_N\in E_{\rm gap}(K)$, $\|f_N\|_{\Lip}\le\|f\|_{\Lip}$, and
\[
        \Delta_{I_j}f_N\longrightarrow \Delta_{I_j}f
        \qquad(j\in\N).
\]
Thus the family $(\mu_{I_j})$ is $c_0$-admissible. The conclusion follows
from Proposition~\ref{prop:admissible}.
\end{proof}

We now compare the gap space with Dalet's predual. For a proper pointed
metric space $M$, Dalet considers the space
\[
        S(M)=\left\{ f\in\lip_0(M):
        \lim_{r\to\infty}
        \sup_{\substack{x\text{ or }y\notin B(0,r)\\ x\ne y}}
        \frac{|f(x)-f(y)|}{d(x,y)}=0
        \right\}.
\]
When $M$ is bounded, the second condition is void. Dalet proved that, for
countable proper $M$, $\F(M)$ is isometric to $S(M)^*$ and has the metric
approximation property \cite{Dal15b}. The following result identifies
$S(K)$ with the gap $c_0$-space for countable proper subsets of the real
line.

\begin{theorem}[Gap $c_0$ coordinates and Dalet's predual]
\label{thm:gap-dalet}
Let $K\subset[0,\infty)$ be infinite, countable, closed, and contain $0$. Then
\[
        E_{\rm gap}(K)=S(K)
\]
isometrically. Consequently,
\[
        S(K)^*\cong \F(K),
        \qquad
        S(K)^{**}\cong \Lip_0(K)
\]
canonically and isometrically.
\end{theorem}

\begin{proof}
Let $f\in\Lip_0(K)$ and set $\alpha_I=\Delta_I f$. We shall use the
following elementary identity. If $x<y$ are points of $K$, then
\[
        f(y)-f(x)=
        \sum_{\substack{I\in\mathcal G(K)\\ I\subset (x,y)}}
        \alpha_I |I|,
        \qquad
        y-x=
        \sum_{\substack{I\in\mathcal G(K)\\ I\subset (x,y)}} |I|.
\]

Both equalities follow from the $L^1$ representation in Lemma~5.1.
For the first, one writes $f(y)-f(x)=\langle f,\delta_K(y)-\delta_K(x)\rangle$
and expands $\delta_K(y)-\delta_K(x)=\sum_{I\subset(x,y)}|I|\,\mu_I$
in $F(K)\cong\ell^1(G(K))$ (Lemma~5.1).
For the second, note that $K$ has Lebesgue measure zero
(it is countable), so $[x,y]$ is covered up to a null set
by the disjoint gaps $I\in G(K)$ with $I\subset(x,y)$,
giving $y-x=\sum_{I\subset(x,y)}|I|$..

Assume first that $f\in E_{\rm gap}(K)$, so $(\alpha_I)$ belongs to $c_0$.
We prove that $f\in S(K)$. Let $\varepsilon>0$ and choose a finite set
$\mathcal F\subset\mathcal G(K)$ such that
\[
        |\alpha_I|<\varepsilon
        \qquad(I\notin\mathcal F).
\]
To prove the little Lipschitz condition, let
\[
        \delta<\min\{|I|:I\in\mathcal F\}
\]
if $\mathcal F\ne\varnothing$, and take any $\delta>0$ otherwise. If
$x<y$ in $K$ and $y-x<\delta$, then no interval of $\mathcal F$ is contained
in $(x,y)$. The identity above gives
\[
        \frac{|f(y)-f(x)|}{y-x}
        \le
        \sup_{\substack{I\in\mathcal G(K)\\ I\subset(x,y)}}|
        \alpha_I|
        \le \varepsilon.
\]
Thus $f\in\lip_0(K)$.

It remains to prove the vanishing condition at infinity. If $K$ is bounded,
there is nothing to prove. Suppose $K$ is unbounded. Choose the finite set
$\mathcal F$ so that $|\alpha_I|<\varepsilon/2$ for $I\notin\mathcal F$.
Let
\[
        A=1+\max\{b_I:I\in\mathcal F\}
\]
with $A=1$ if $\mathcal F=\varnothing$, and put
\[
        C=\sum_{I\in\mathcal F} |\alpha_I| |I|.
\]
Choose $r>A$ so large that $C/(r-A)<\varepsilon/2$. Let $x<y$ in $K$ with
$y>r$. If $x\ge A$, then $(x,y)$ contains no interval of $\mathcal F$, and
the quotient is at most $\varepsilon/2$. If $x<A$, then the contribution of
intervals outside $\mathcal F$ is at most $(\varepsilon/2)(y-x)$, while the
contribution of intervals in $\mathcal F$ is at most $C$. Hence
\[
        \frac{|f(y)-f(x)|}{y-x}
        \le \frac{\varepsilon}{2}+\frac{C}{y-x}
        \le \frac{\varepsilon}{2}+\frac{C}{r-A}
        <\varepsilon.
\]
Therefore $f\in S(K)$.

Conversely, assume $f\in S(K)$. We prove that $(\alpha_I)\in c_0$. Let
$\varepsilon>0$. If $K$ is unbounded, choose $r>0$ such that
\[
        \frac{|f(x)-f(y)|}{|x-y|}<\varepsilon
\]
whenever $x\ne y$ and $x$ or $y$ lies outside $B(0,r)$. If $K$ is bounded,
choose $r>R_K$. Since $f\in\lip_0(K)$, choose $\delta>0$ such that
\[
        0<|x-y|<\delta
        \quad\Longrightarrow\quad
        \frac{|f(x)-f(y)|}{|x-y|}<\varepsilon.
\]
If $I=(a_I,b_I)$ satisfies $b_I>r$, then $|\alpha_I|<\varepsilon$ by the
vanishing condition at infinity. If $b_I\le r$ and $|I|<\delta$, then
$|\alpha_I|<\varepsilon$ by the little Lipschitz condition. Hence the only
gaps with $|\alpha_I|\ge\varepsilon$ are contained in $[0,r]$ and have
length at least $\delta$. There are only finitely many such disjoint
intervals. Thus $(\alpha_I)\in c_0$.

We have proved $E_{\rm gap}(K)=S(K)$. The duality and biduality statements
now follow from Proposition~\ref{prop:gap-admissible}.
\end{proof}

\begin{example}
Let
\[
        K=\{0\}\cup\left\{\frac1n:n\in\N\right\}
        \subset[0,1].
\]
The gaps are $(1/(n+1),1/n)$, $n\in\N$. Theorem~\ref{thm:gap-dalet} says
that the Dalet predual is the space of all $f\in\Lip_0(K)$ such that
\[
        \frac{f(1/n)-f(1/(n+1))}{1/n-1/(n+1)}\longrightarrow 0.
\]
Thus the correct $c_0$ coordinates near the accumulation point are adjacent
gap slopes, not the long molecules $m_{1/n,0}$.
\end{example}

\begin{remark}
Theorem~\ref{thm:gap-dalet} recovers Dalet's predual only for countable
proper subsets of the real line. It does not recover Dalet's theorem for
arbitrary countable proper metric spaces. In general, selecting a
$c_0$-coordinate family must reflect the accumulation structure of the
space. The gap coordinates provide such a selection in one dimension; for
more general countable proper spaces, this becomes a separate coordinate
selection problem and is not pursued here.
\end{remark}

\section{The proper discrete case}

We finish by recording the discrete molecular case in the language of the
preceding sections. It is included mainly as a comparison with Dalet's
predual.

\begin{proposition}\label{prop:proper-discrete}
Let $(M,d,0)$ be an infinite pointed proper discrete metric space. Let
$\{(p_n,q_n):n\in\N\}$ be an enumeration of all ordered pairs of distinct
points of $M$, and set
\[
        \mu_n=m_{p_n,q_n}
        =\frac{\delta(p_n)-\delta(q_n)}{d(p_n,q_n)}.
\]
Define
\[
        E_{\rm mol}(M)=\left\{f\in\Lip_0(M):
        (\langle f,
        \mu_n\rangle)_{n\in\N}\in c_0\right\}.
\]
Then
\[
        E_{\rm mol}(M)^*\cong \F(M),
        \qquad
        E_{\rm mol}(M)^{**}\cong \Lip_0(M)
\]
canonically and isometrically. Moreover $E_{\rm mol}(M)$ coincides with
Dalet's predual $S(M)$.
\end{proposition}

\begin{proof}
Since $M$ is proper and discrete, closed balls are finite and $M$ is
countable. The family of all molecules is a metric quotient coordinate
family: the Arens--Eells description of the free norm says that
\[
        \ell^1\ni (a_n)\longmapsto \sum_{n=1}^\infty a_n m_{p_n,q_n}
        \in\F(M)
\]
is a metric quotient map onto $\F(M)$; see, for instance, \cite{GK03,Wea18}.

Every compact subset of $M$ is finite. If $h\in\Lip_0(M)$ has finite
support and $\varepsilon>0$, then only finitely many pairs $(x,y)$ satisfy
\[
        \frac{|h(x)-h(y)|}{d(x,y)}\ge\varepsilon.
\]
Indeed, at least one of $x,y$ must belong to $\supp(h)$, and the other must
lie in a ball of radius $2\|h\|_\infty/\varepsilon$ around that finite
support. These balls are finite. Hence
$\Lip_{0,c}(M)\subset E_{\rm mol}(M)$. Corollary~\ref{cor:compact-admissible}
therefore gives the two duality assertions.

It remains to identify the space with $S(M)$. The condition
$f\in E_{\rm mol}(M)$ says that for every $\varepsilon>0$ there are only
finitely many ordered pairs $(x,y)$ with
\[
        \frac{|f(x)-f(y)|}{d(x,y)}\ge\varepsilon.
\]
This immediately implies the little Lipschitz condition and the vanishing
condition at infinity in Dalet's definition. Conversely, if $f\in S(M)$,
then large quotients can occur only inside a fixed closed ball and at
scales bounded below. Such a ball is finite, because $M$ is proper and
discrete; hence there are only finitely many corresponding pairs. Thus
$E_{\rm mol}(M)=S(M)$.
\end{proof}

\begin{remark}
Dalet's theorem applies to all countable proper metric spaces, and therefore
already implies that $\F(M)$ is a dual space in the proper discrete case.
Proposition~\ref{prop:proper-discrete} should not be read as a new duality
theorem in that setting. Its role is to identify Dalet's predual with the
molecular $c_0$-coordinate space and to show how this identification follows
from the scalar criterion.
\end{remark}

\section*{Acknowledgements and funding}
The author acknowledges the support of INdAM--GNAMPA Project, CUP
E53C25002010001, and PRIN 2022 ``Advanced theoretical aspects in PDEs and
their applications'', project 2022BCFHN2, CUP I53D23002300006.

\end{document}